\pdfoutput=1 

\documentclass[11pt]{amsart}

\usepackage{amssymb,amsthm,enumitem,colonequals,tikz-cd,microtype}
\usepackage[normalem]{ulem} 

\usepackage[osf]{Baskervaldx}
\usepackage[baskervaldx]{newtxmath}
\usepackage[cal=boondoxo]{mathalfa} 

\usepackage[top=3.75cm, bottom=3cm, left=3.5cm, right=3.5cm]{geometry}

\usepackage{xcolor}
\colorlet{darkblue}{blue!55!black}
\colorlet{darkcyan}{cyan!50!black}
\colorlet{darkgreen}{green!60!black}

\PassOptionsToPackage{hyphens}{url}
\usepackage{hyperref}
\hypersetup{
    colorlinks=true,
    linkcolor=darkblue,
    urlcolor=darkcyan,
    citecolor=darkgreen,
}

\def\eqref#1{\textcolor{darkblue}{(\ref{#1})}}

\usepackage[nameinlink]{cleveref} 
\Crefformat{section}{#2\S#1#3}
\Crefmultiformat{section}{#2\S\S#1#3}{ and~#2#1#3}{, #2#1#3}{, and~#2#1#3}

\usepackage[pagewise]{lineno}
\let\oldequation\equation
\let\oldendequation\endequation
\renewenvironment{equation}{\linenomathNonumbers\oldequation}{\oldendequation\endlinenomath}
\expandafter\let\expandafter\oldequationstar\csname equation*\endcsname
\expandafter\let\expandafter\oldendequationstar\csname endequation*\endcsname
\renewenvironment{equation*}{\linenomathNonumbers\oldequationstar}{\oldendequationstar\endlinenomath}
\let\oldalign\align
\let\oldendalign\endalign
\renewenvironment{align}{\linenomathNonumbers\oldalign}{\oldendalign\endlinenomath}
\expandafter\let\expandafter\oldalignstar\csname align*\endcsname
\expandafter\let\expandafter\oldendalignstar\csname endalign*\endcsname
\renewenvironment{align*}{\linenomathNonumbers\oldalignstar}{\oldendalignstar\endlinenomath}

\theoremstyle{plain}
\newtheorem{theorem}{Theorem}[section]
\newtheorem{lemma}[theorem]{Lemma}

\newtheorem{proposition}[theorem]{Proposition}

\theoremstyle{definition}

\newtheorem{remark}[theorem]{Remark}
\newtheorem{setup}[theorem]{Setup}

\newtheorem*{ack}{Acknowledgments}

\AddToHook{env/conjecture/begin}{\crefalias{theorem}{conjecture}}
\AddToHook{env/lemma/begin}{\crefalias{theorem}{lemma}}
\AddToHook{env/corollary/begin}{\crefalias{theorem}{corollary}}
\AddToHook{env/proposition/begin}{\crefalias{theorem}{proposition}}
\AddToHook{env/definition/begin}{\crefalias{theorem}{definition}}
\AddToHook{env/remark/begin}{\crefalias{theorem}{remark}}
\AddToHook{env/setup/begin}{\crefalias{theorem}{setup}}
\AddToHook{env/example/begin}{\crefalias{theorem}{example}}
\AddToHook{env/hypothesis/begin}{\crefalias{theorem}{hypothesis}}

\numberwithin{equation}{section}
\numberwithin{theorem}{section}

\title[A blowup criterion for regularity]{A blowup criterion for regularity}

\author[P.~Lank]{Pat Lank}
\address{P.~Lank,
Dipartimento di Matematica “F. Enriques”, Universit\`{a} degli Studi di Milano, Via Cesare
Saldini 50, 20133 Milano, Italy}
\email{plankmathematics@gmail.com}

\date{\today}



\begin{document}
    
\begin{abstract}
    We prove that a Noetherian integral domain is regular if, and only if, the derived pushforward of the structure sheaf along every blowup is a perfect complex. 
    For suitable local rings, we give a construction which realizes a shift of the residue field as a direct summand of the derived pushforward of the structure sheaf along an explicit blowup.

\end{abstract}

\maketitle

\tableofcontents

\section{Introduction}
\label{sec:intro}

\subsection{Background}
\label{sec:intro_background}

In \cite{Ma/Schwede:2020}, Ma and Schwede gave a characterization of regularity.
Let $R$ be a local domain essentially of finite type over a field of characteristic zero.
Set $X:=\operatorname{Spec}(R)$.
They showed that $R$ is regular if, and only if, $\mathbf{R}f_\ast \mathcal{O}_Y \in \operatorname{Perf}(X)$ for all regular alterations $f\colon Y \to X$.

Their approach uses alterations and builds on the relationship between multiplier ideals and test ideals \cite{Blickle/Schwede/Tucker:2015}.
Also, their work is guided by the connections between big Cohen--Macaulay algebras and resolutions of singularities.
See \cite[Proposition 3.20]{Iyengar/Ma/Schwede/Walker:2021} and \cite[Theorem 3.3]{Roberts:1980}.

In \cite{Kunz:1969}, Kunz showed that a Noetherian ring $R$ of prime characteristic is regular if, and only if, its Frobenius morphism $F\colon R \to R$ is flat. 
The prime characteristic analogs of Ma and Schwede's result \cite{Ma/Schwede:2020} follow from Bhatt's work on killing cohomology by finite covers \cite{Bhatt:2012} and the characterizations of regularity by absolute integral closures \cite{Bhatt/Iyengar/Ma:2019,Aberbach/Li:2008}.

In fact, Bhatt--Iyengar--Ma \cite{Bhatt/Iyengar/Ma:2019} established a $p$-adic analog of Kunz's result using perfectoid algebras.
More recently, Bhatt has shown that regularity of a Noetherian integral affine scheme $X$ can be detected by the existence of a cofinal system of alterations $f\colon Y \to X$ satisfying $\mathbf{R}f_\ast \mathcal{O}_Y \in \operatorname{Perf}(X)$.
See \cite[Theorem 8.5.1]{Bhatt:2025}.

\subsection{Main result}
\label{sec:intro_main_result}

Bhatt's result answers \cite[Question 1.1]{Ma/Schwede:2020} affirmatively and applies more generally to Noetherian integral domains.
Ma and Schwede asked whether regularity can be characterized by the perfectness of derived pushforwards of structure sheaves for regular alterations of rings essentially of finite type over the integers.
In this setting, by the work of de Jong and Gabber (see e.g.\ \cite[Expos\'{e} 0, Th\'{e}or\`{e}me 3(2)]{Illusie/Laszlo/Orgogozo:2014}), there exists a cofinal system of regular alterations.

We give a complementary criterion for regularity.

\begin{theorem}
    \label{thm:characterize_regularity}
    For every Noetherian integral affine scheme $X=\operatorname{Spec}(R)$, the following are equivalent:
    \begin{enumerate}
        \item \label{cor:characterize_regularity1} $X$ is regular
        \item \label{cor:characterize_regularity3} $\mathbf{R}f_\ast \mathcal{O}_{X^\prime} \in \operatorname{Perf}(X)$ for all blowups $f\colon X^\prime \to X$.
    \end{enumerate} 
\end{theorem}

\subsection{Comparisons}
\label{sec:intro_comparison}

\cite[Theorem 8.5.1]{Bhatt:2025} shows that regularity of a Noetherian integral domain is detected by a cofinal system of alterations whose derived pushforwards of structure sheaves are perfect. 
We show that testing all blowups suffices. 
Blowups along nonzero ideals are alterations, but are not cofinal among alterations in general.

\Cref{thm:characterize_regularity} should also be compared with the characterization of regularity by quasi-perfectness of blowups at closed points \cite[Proposition A.3]{DeDeyn/Lank/ManaliRahul:2025b}.
Quasi-perfectness requires that the derived pushforward of any perfect complex remain perfect.
In contrast, we test only the structure sheaf, at the cost of considering all blowups.

\subsection{Auxiliary result}
\label{sec:intro_auxiliary_result}

Our proof of \Cref{thm:characterize_regularity} proceeds by an explicit construction with blowups. 
The perfectoid methods of \cite{Bhatt/Iyengar/Ma:2019,Bhatt:2025} do not appear in our work.
We discuss this below.

The main difficulty is to deduce regularity from the perfectness hypothesis.
We first establish normality using \Cref{lem:perfect_finite_modification}.
This gives an isomorphism criterion for finite birational morphisms in terms of perfectness of pushforward of the structure sheaf.

We establish the following auxiliary result.

\begin{proposition}
    \label{prop:blowup_extensions_for_residue_field}
    Let $(R,\mathfrak{m},k)$ be a Noetherian local domain which admits an $R$-regular sequence $a,b\in \mathfrak{m}$.
    Define the ideal $J:=(a^3,ab^2,b^3)+ a^2 b \mathfrak{m}$.
    Set $X=\operatorname{Spec}(R)$ and $i\colon \operatorname{Spec}(k) \to X$ to be the associated closed immersion.
    Denote by $f_{\mathcal{J}} \colon \underline{\operatorname{Proj}}_X (\bigoplus^\infty_{n=0} \mathcal{J}^n) \to X$ the blowup of $X$ along $\mathcal{J}$.
    There exists an isomorphism
    \begin{displaymath}
        \mathbf{R}(f_\mathcal{J})_\ast \mathcal{O}_{\underline{\operatorname{Proj}}_X (\bigoplus^\infty_{n=0} \mathcal{J}^n)} \cong \mathbf{R}i_\ast \mathcal{O}_{\operatorname{Spec}(k)}[-1] \oplus \mathcal{O}_X.
    \end{displaymath}
    Moreover, $X$ is regular if, and only if, $\mathbf{R}(f_\mathcal{J})_\ast \mathcal{O}_{\underline{\operatorname{Proj}}_X (\bigoplus^\infty_{n=0} \mathcal{J}^n)} \in \operatorname{Perf}(X)$.
\end{proposition}

\Cref{prop:blowup_extensions_for_residue_field} shows that regularity of $R$ is controlled by a single explicit blowup.
Moreover, the shifted residue field is the obstruction to the perfectness of $\mathbf{R}(f_\mathcal{J})_\ast \mathcal{O}_{\underline{\operatorname{Proj}}_X (\bigoplus^\infty_{n=0} \mathcal{J}^n)}$. 
This is akin to the Auslander--Buchsbaum--Serre criterion for regularity.

Back to the proof of \Cref{thm:characterize_regularity}.
We use that every ideal of a localization is extended from the original ring. 
Compatibility of blowups and derived pushforwards with flat base change shows the hypothesis in \Cref{thm:characterize_regularity} holds for the local rings. 
After normality is established, local rings of Krull dimension at most are regular.
In higher Krull dimension, \Cref{prop:blowup_extensions_for_residue_field} and the Auslander--Buchsbaum--Serre criterion apply.

\begin{ack}
    The author was supported by the ERC Advanced Grant 101095900-TriCatApp, and thanks Bhargav Bhatt and Karl Schwede for helpful discussions.
\end{ack}

\section{Preliminaries}
\label{sec:prelim}

We follow the conventions of \cite{stacks-project}.
Let $X$ be a Noetherian scheme. 

\subsection{Categories}
\label{sec:prelim_categories}

$\operatorname{Mod}(X)$ is the Grothendieck abelian category of sheaves of $\mathcal{O}_X$-modules on the small Zariski site $X_{\textrm{Zar}}$ of $X$.
$\operatorname{Qcoh}(X)$ is the strictly full subcategory of $\operatorname{Mod}(X)$ consisting of quasi-coherent sheaves. 
$D(X) := D(\operatorname{Mod}(X))$ is the derived category of $\operatorname{Mod}(X)$. 
$D_{\operatorname{qc}}(X)$ is the strictly full subcategory of $D(X)$ consisting of complexes with quasi-coherent cohomology sheaves. 
$\operatorname{Perf}(X)$ is the strictly full subcategory of perfect complexes in $D_{\operatorname{qc}}(X)$. 
$D^b_{\operatorname{coh}}(X)$ is the strictly full subcategory of $D_{\operatorname{qc}}(X)$ consisting of bounded pseudocoherent complexes. 
When $X$ is affine, we will sometimes abuse notation by writing $D^b_{\operatorname{coh}}(R):=D^b_{\operatorname{coh}}(X)$ where $R:=H^0(X,\mathcal{O}_X)$.
Similar conventions will occur for other categories or numerical invariants.

\subsection{Support}
\label{sec:prelim_support}

For $E\in D_{\operatorname{qc}}(X)$, set $\operatorname{Supp}(E) := \cup_{i\in \mathbb{Z}} \operatorname{Supp}(\mathcal{H}^i (E))$.
If $Z\subseteq X$ is a closed subset, then $D_{\operatorname{qc},Z}(X)$ is the strictly full subcategory of $D_{\operatorname{qc}}(X)$ consisting of objects $E$ such that $\operatorname{Supp}(E) \subseteq Z$.
Similar notation is defined for subcategories of $D^b_{\operatorname{coh}}(X)$. 

\subsection{Ideal-theoretic reductions}
\label{sec:prelim_categories_reductions}

If $\mathcal{P}\subseteq \mathcal{O}_X$ is an ideal sheaf, we let 
\begin{displaymath}
    f_\mathcal{P} \colon \underline{\operatorname{Proj}}_X (\bigoplus^\infty_{n=0} \mathcal{P}^n) \to X
\end{displaymath}
be the blowup of $X$ along $\mathcal{P}$.
Given a pair of ideal sheaves $\mathcal{P}\subseteq \mathcal{J}\subseteq \mathcal{O}_X$, we say that $\mathcal{P}$ is a \textbf{reduction} of $\mathcal{J}$ if $\mathcal{P}\mathcal{J}^N = \mathcal{J}^{N+1}$ for some $N>0$.

\subsection{Regular sequences}
\label{sec:prelim_regular_sequences}

Let $(R,\mathfrak{m})$ be a Noetherian local ring. 
A sequence $r_1,\ldots,r_n \in R$ is called an \textbf{$R$-regular sequence} if $r_i$ is a nonzerodivisor on $R/(r_1,\ldots, r_{i-1})$ for each $i=1,\ldots, n$ and the $R$-module $R/(r_1,\ldots,r_n)$ is nonzero.
Some authors omit the condition that $R/(r_1,\ldots,r_n)$ is not zero.
We follow \cite[\href{https://stacks.math.columbia.edu/tag/00LF}{Tag 00LF}]{stacks-project}.
With this convention, $R$-regular sequences need not be preserved under localization.
The \textbf{depth} of $R$ is the supremum in $\{0,1,\ldots,\infty\}$ of the lengths of $R$-regular sequences in $\mathfrak{m}$.
See \cite[\href{https://stacks.math.columbia.edu/tag/00LE}{Tag 00LE}]{stacks-project}.

\section{Blowups}
\label{sec:blowup}

Throughout, we use the following notation.

\begin{setup}
    \label{setup:main}
    Suppose $(R,\mathfrak{m},k)$ is a Noetherian local domain which admits an $R$-regular sequence $a,b\in \mathfrak{m}$. 
    Let $X:=\operatorname{Spec}(R)$, $i \colon \operatorname{Spec}(k) \to X$ be the associated closed immersion, and $c:=a^2 b$.
    Choose a minimal set of generators $r_1,\ldots,r_e$ for $\mathfrak{m}$.
    Define the ideals $P=(a,b)^3$ and $J=(a^3,ab^2,b^3)+ a^2 b \mathfrak{m}$ of $R$.
    Denote by $\mathcal{P}$ and $\mathcal{J}$ respectively for the sheafifications of $P$ and $J$.
\end{setup}

\begin{lemma}
    \label{lem:reduction_ideals}
    Assume \Cref{setup:main}. 
    Then $J$ is a reduction of $P$.
\end{lemma}

\begin{proof}
    We can express $P$ as
    \begin{displaymath}
        P = (a,b)^3 = (a^3, c, ab^2, b^3) = J + (c).
    \end{displaymath}
    Since $c^2 = a^3 (ab^2)\in J^2$, it follows that 
    \begin{displaymath}
        \begin{aligned}
            P^2
            &= (J+(c))^2 
            \\&= J^2 + (c)J + (c^2) 
            \\&= J^2 + (c) J
            \\&= J (J+(c))
            \\&= JP.
        \end{aligned}
    \end{displaymath}
    Thus, $J$ is a reduction of $P$.
\end{proof}

\begin{lemma}
    \label{lem:finite_birational}
    Assume \Cref{setup:main}. 
    There exists a commutative diagram
    \begin{equation}
        \label{eq:finite_birational}
        \begin{tikzcd}
            {\underline{\operatorname{Proj}}_X (\bigoplus^\infty_{n=0} \mathcal{P}^n)} & {\underline{\operatorname{Proj}}_X (\bigoplus^\infty_{n=0} \mathcal{J}^n)} \\
            & {X}
            \arrow["v", from=1-1, to=1-2]
            \arrow["{f_{\mathcal{P}}}"', from=1-1, to=2-2]
            \arrow["{f_{\mathcal{J}}}", from=1-2, to=2-2]
        \end{tikzcd}
    \end{equation}
    where $v$ is a finite birational morphism.
\end{lemma}

\begin{proof}
    By \Cref{lem:reduction_ideals}, $J$ is a reduction of $P$.
    Then \cite[Lemma 4.2]{Abad/Bravo/Villamayor:2020} says there exists a finite morphism $v$ yielding the desired diagram. 
    As $P$ and $J$ are nonzero ideals, integrality of $X$ implies both blowups are integral \cite[\href{https://stacks.math.columbia.edu/tag/02ND}{Tag 02ND}]{stacks-project}.
    Then $f_{\mathcal{P}},f_{\mathcal{J}}$ are respectively isomorphisms over $X\setminus V(P)$ and $X\setminus V(J)$. 
    Consequently, $J\subseteq P$ implies $V(P)\subseteq V(J)$, where $V(-)$ denotes the vanishing locus of an ideal.
    Hence, $v$ restricts to an $X$-isomorphism over $f_{\mathcal{J}}^{-1}(X\setminus V(J))$. 
    Thus, $v$ is birational \cite[\href{https://stacks.math.columbia.edu/tag/0BAD}{Tag 0BAD}]{stacks-project}.
\end{proof}

\begin{remark}
    \label{rmk:open_cover}
    Since $\mathfrak{m}= (r_1,\ldots,r_e)$, it follows that $J = (a^3,ab^2,b^3,c r_1, \ldots, c r_e)$. 
    By \cite[\href{https://stacks.math.columbia.edu/tag/0804}{Tag 0804}]{stacks-project}, $\underline{\operatorname{Proj}}_X (\oplus^\infty_{n=0} \mathcal{J}^n)$ admits an affine open cover given by $\operatorname{Spec}(R[\frac{J}{d}])$ where $d\in \{a^3,ab^2,b^3,c r_1, \ldots, c r_e\}$. 
    Here $R[\frac{J}{d}]\subseteq R_d$ is the $R$-subalgebra generated by $\frac{j}{d}$ with $j\in J$.
    As $J$ is a reduction of $P$, \cite[Remark 4.3]{Abad/Bravo/Villamayor:2020} implies that
    \begin{displaymath}
        v^{-1} (\operatorname{Spec}(R[\frac{J}{d}])) = \operatorname{Spec}(R[\frac{P}{d}]).
    \end{displaymath}
    For each $d\in \{a^3,ab^2,b^3,c r_1, \ldots, c r_e\}$, there exists a fibered square
    \begin{equation}
        \label{eq:open_cover}
        \begin{tikzcd}
            {\operatorname{Spec}(R[\frac{P}{d}])} & {\operatorname{Spec}(R[\frac{J}{d}])} \\
            {\underline{\operatorname{Proj}}_X (\bigoplus^\infty_{n=0} \mathcal{P}^n)} & {\underline{\operatorname{Proj}}_X (\bigoplus^\infty_{n=0} \mathcal{J}^n)}
            \arrow["{v^d}", from=1-1, to=1-2]
            \arrow["{u^P_d}"', from=1-1, to=2-1]
            \arrow["{u_d^J}", from=1-2, to=2-2]
            \arrow["v"', from=2-1, to=2-2]
        \end{tikzcd}
    \end{equation}
    where $u^P_d$ and $u^J_d$ are the associated open immersions.
\end{remark}

\begin{lemma}
    \label{lem:concentrated_degree_zero_global}
    Assume \Cref{setup:main}. 
    There exists an isomorphism
    \begin{displaymath}
        \begin{aligned}
            \operatorname{coker}
            & (\mathcal{O}_{\underline{\operatorname{Proj}}_X (\bigoplus^\infty_{n=0} \mathcal{J}^n)} \to v_\ast \mathcal{O}_{\underline{\operatorname{Proj}}_X (\bigoplus^\infty_{n=0} \mathcal{P}^n)})
            \\&\cong \operatorname{cone}(\mathcal{O}_{\underline{\operatorname{Proj}}_X (\bigoplus^\infty_{n=0} \mathcal{J}^n)} \to \mathbf{R}v_\ast \mathcal{O}_{\underline{\operatorname{Proj}}_X (\bigoplus^\infty_{n=0} \mathcal{P}^n)}).
        \end{aligned}
    \end{displaymath}
    Moreover, for each $d\in \{a^3,ab^2,b^3,c r_1, \ldots, c r_e\}$, there exist isomorphisms
    \begin{equation}
        \label{eq:concentrated_degree_zero_global}
        \begin{aligned}
            \mathbf{L} (u^J_d)^\ast 
            & \operatorname{cone}(\mathcal{O}_{\underline{\operatorname{Proj}}_X (\bigoplus^\infty_{n=0} \mathcal{J}^n)} \to \mathbf{R}v_\ast \mathcal{O}_{\underline{\operatorname{Proj}}_X (\bigoplus^\infty_{n=0} \mathcal{P}^n)}) 
            \\&\cong \operatorname{cone} (\mathcal{O}_{\operatorname{Spec}(R[\frac{J}{d}])} \to \mathbf{R} v^d_\ast \mathcal{O}_{\operatorname{Spec}(R[\frac{P}{d}])})
            \\&\cong \operatorname{coker}(\mathcal{O}_{\operatorname{Spec}(R[\frac{J}{d}])} \to v^d_\ast \mathcal{O}_{\operatorname{Spec}(R[\frac{P}{d}])})
            \\&\cong (u^J_d)^\ast  \operatorname{coker}(\mathcal{O}_{\underline{\operatorname{Proj}}_X (\bigoplus^\infty_{n=0} \mathcal{J}^n)} \to v_\ast \mathcal{O}_{\underline{\operatorname{Proj}}_X (\bigoplus^\infty_{n=0} \mathcal{P}^n)}).
        \end{aligned}
    \end{equation}
    The maps for these cones and cokernels are taken to be the corresponding adjunction units.
\end{lemma}

\begin{proof}
    By \Cref{lem:finite_birational}, $v$ is finite and birational.
    Hence, \cite[\href{https://stacks.math.columbia.edu/tag/0G9R}{Tag 0G9R}]{stacks-project} yields that the canonical morphism 
    \begin{displaymath}
        v_\ast \mathcal{O}_{\underline{\operatorname{Proj}}_X (\bigoplus^\infty_{n=0} \mathcal{P}^n)} \to \mathbf{R}v_\ast \mathcal{O}_{\underline{\operatorname{Proj}}_X (\bigoplus^\infty_{n=0} \mathcal{P}^n)}
    \end{displaymath}
    is an isomorphism.
    Moreover, \cite[\href{https://stacks.math.columbia.edu/tag/0CC1}{Tag 0CC1}]{stacks-project} implies 
    \begin{displaymath}
        \mathcal{O}_{\underline{\operatorname{Proj}}_X (\bigoplus^\infty_{n=0} \mathcal{J}^n)} \xrightarrow{unit} v_\ast \mathcal{O}_{\underline{\operatorname{Proj}}_X (\bigoplus^\infty_{n=0} \mathcal{P}^n)}
    \end{displaymath}
    is injective.
    This gives a short exact sequence, and hence a distinguished triangle,
    \begin{displaymath}
        \begin{aligned}
            \mathcal{O}_{\underline{\operatorname{Proj}}_X (\bigoplus^\infty_{n=0} \mathcal{J}^n)} 
            &\xrightarrow{unit} v_\ast \mathcal{O}_{\underline{\operatorname{Proj}}_X (\bigoplus^\infty_{n=0} \mathcal{P}^n)} 
            \\&\to \operatorname{coker}(\mathcal{O}_{\underline{\operatorname{Proj}}_X (\bigoplus^\infty_{n=0} \mathcal{J}^n)} \to v_\ast \mathcal{O}_{\underline{\operatorname{Proj}}_X (\bigoplus^\infty_{n=0} \mathcal{P}^n)})
            \to \mathcal{O}_{\underline{\operatorname{Proj}}_X (\bigoplus^\infty_{n=0} \mathcal{J}^n)}[1].
        \end{aligned}
    \end{displaymath}
    This distinguished triangle is isomorphic to the distinguished triangle 
    \begin{displaymath}
        \begin{aligned}
            \mathcal{O}_{\underline{\operatorname{Proj}}_X (\bigoplus^\infty_{n=0} \mathcal{J}^n)} 
            &\xrightarrow{unit} \mathbf{R} v_\ast \mathcal{O}_{\underline{\operatorname{Proj}}_X (\bigoplus^\infty_{n=0} \mathcal{P}^n)} 
            \\&\to \operatorname{cone}(\mathcal{O}_{\underline{\operatorname{Proj}}_X (\bigoplus^\infty_{n=0} \mathcal{J}^n)} \to \mathbf{R} v_\ast \mathcal{O}_{\underline{\operatorname{Proj}}_X (\bigoplus^\infty_{n=0} \mathcal{P}^n)})
            \to \mathcal{O}_{\underline{\operatorname{Proj}}_X (\bigoplus^\infty_{n=0} \mathcal{J}^n)}[1].
        \end{aligned}
    \end{displaymath}
    Indeed, the two canonically composed morphisms 
    \begin{displaymath}
        \mathbf{L}v^\ast \mathcal{O}_{\underline{\operatorname{Proj}}_X (\bigoplus^\infty_{n=0} \mathcal{J}^n)} 
        \to v^\ast \mathcal{O}_{\underline{\operatorname{Proj}}_X (\bigoplus^\infty_{n=0} \mathcal{J}^n)} 
        \to \mathcal{O}_{\underline{\operatorname{Proj}}_X (\bigoplus^\infty_{n=0} \mathcal{P}^n)} 
    \end{displaymath}
    and 
    \begin{displaymath}
        \mathcal{O}_{\underline{\operatorname{Proj}}_X (\bigoplus^\infty_{n=0} \mathcal{J}^n)} 
        \to  v_\ast \mathcal{O}_{\underline{\operatorname{Proj}}_X (\bigoplus^\infty_{n=0} \mathcal{P}^n)}
        \to \mathbf{R} v_\ast \mathcal{O}_{\underline{\operatorname{Proj}}_X (\bigoplus^\infty_{n=0} \mathcal{P}^n)}
    \end{displaymath}
    correspond under the adjunction isomorphism of derived pullback and pushforward. 
    See \cite[Eq.\ (3.2.1.3) in proof of Proposition 3.2.1 and Exercise 3.2.5(a)]{Lipman/Hashimoto:2009}. 
    This yields the first claim.

    We prove the second claim.
    By \cite[\S 1, Proposition 3.9.9]{Grothendieck/Dieudonne:1971}, birationality is preserved under flat base change. 
    Hence, $v^d$ is birational for $d\in \{a^3,ab^2,b^3,c r_1, \ldots, c r_e\}$.
    Finiteness is preserved under base change, so each $v^d$ is finite.
    Then a similar argument in the first claim above yields 
    \begin{displaymath}
        \begin{aligned}
            \operatorname{cone}
            & (\mathcal{O}_{\operatorname{Spec}(R[\frac{J}{d}])} \to \mathbf{R} v^d_\ast \mathcal{O}_{\operatorname{Spec}(R[\frac{P}{d}])})
            \\&\cong \operatorname{coker}(\mathcal{O}_{\operatorname{Spec}(R[\frac{J}{d}])} \to v^d_\ast \mathcal{O}_{\operatorname{Spec}(R[\frac{P}{d}])}).
        \end{aligned}
    \end{displaymath}
    Applying $\mathbf{L}(u^J_d)^\ast$ to the first claim gives us 
    \begin{displaymath}
        \begin{aligned}
            \mathbf{L}(u^J_d)^\ast \operatorname{coker}
            & (\mathcal{O}_{\underline{\operatorname{Proj}}_X (\bigoplus^\infty_{n=0} \mathcal{J}^n)} \to v_\ast \mathcal{O}_{\underline{\operatorname{Proj}}_X (\bigoplus^\infty_{n=0} \mathcal{P}^n)})
            \\&\cong \mathbf{L}(u^J_d)^\ast \operatorname{cone}(\mathcal{O}_{\underline{\operatorname{Proj}}_X (\bigoplus^\infty_{n=0} \mathcal{J}^n)} \to \mathbf{R}v_\ast \mathcal{O}_{\underline{\operatorname{Proj}}_X (\bigoplus^\infty_{n=0} \mathcal{P}^n)}).
        \end{aligned}
    \end{displaymath}
    Since $u^J_d$ is an open immersion, $\mathbf{L}(u^J_d)^\ast$ is defined by applying $(u^J_d)^\ast$ componentwise for any representative of a complex \cite[\href{https://stacks.math.columbia.edu/tag/02N4}{Tags 02N4} \& \href{https://stacks.math.columbia.edu/tag/015F}{015F}]{stacks-project}.
    Hence, we obtain that 
    \begin{displaymath}
        \begin{aligned}
            \mathbf{L}(u^J_d)^\ast \operatorname{coker}
            & (\mathcal{O}_{\underline{\operatorname{Proj}}_X (\bigoplus^\infty_{n=0} \mathcal{J}^n)} \to v_\ast \mathcal{O}_{\underline{\operatorname{Proj}}_X (\bigoplus^\infty_{n=0} \mathcal{P}^n)})
            \\&\cong (u^J_d)^\ast \operatorname{coker}(\mathcal{O}_{\underline{\operatorname{Proj}}_X (\bigoplus^\infty_{n=0} \mathcal{J}^n)} \to v_\ast \mathcal{O}_{\underline{\operatorname{Proj}}_X (\bigoplus^\infty_{n=0} \mathcal{P}^n)}).
        \end{aligned}
    \end{displaymath}

    To finish the proof, it suffices to check that 
    \begin{displaymath}
        \begin{aligned}
            (u^J_d)^\ast 
            & \operatorname{coker}(\mathcal{O}_{\underline{\operatorname{Proj}}_X (\bigoplus^\infty_{n=0} \mathcal{J}^n)} \to v_\ast \mathcal{O}_{\underline{\operatorname{Proj}}_X (\bigoplus^\infty_{n=0} \mathcal{P}^n)})
            \\&\cong \operatorname{coker}(\mathcal{O}_{\operatorname{Spec}(R[\frac{J}{d}])} \to v^d_\ast \mathcal{O}_{\operatorname{Spec}(R[\frac{P}{d}])}).
        \end{aligned}
    \end{displaymath}
    By flat base change in \eqref{eq:open_cover}, there exists a canonical isomorphism 
    \begin{displaymath}
        \beta^J_d \colon (u^J_d)^\ast v_\ast \mathcal{O}_{\underline{\operatorname{Proj}}_X (\bigoplus^\infty_{n=0} \mathcal{P}^n)} \to (v_d)_\ast \mathcal{O}_{\operatorname{Spec}(R[\frac{P}{d}])}.
    \end{displaymath}
    See \cite[Remark 22.94 \& Theorem 22.99]{Gortz/Wedhorn:2023}.
    For each open $W\subseteq \operatorname{Spec}(R[\frac{J}{d}])$, the source and target of $\beta^J_d$ coincide with $\mathcal{O}_{\underline{\operatorname{Proj}}_X (\bigoplus^\infty_{n=0} \mathcal{P}^n)}(v^{-1}(W))$ because 
    \begin{displaymath}
        v^{-1}_d (W) = v^{-1}(W) \subseteq \operatorname{Spec}(R[\frac{P}{d}]) \subseteq \underline{\operatorname{Proj}}_X (\bigoplus^\infty_{n=0} \mathcal{P}^n)
    \end{displaymath}
    and $(\beta^J_d)(W)$ is the identity. 
    Denote by $\eta_v$ and $\eta_{v_d}$ respectively the units for pullback and pushforward for $v$ and $v_d$. 
    Under the canonical isomorphism $(u^J_d)^\ast \mathcal{O}_{\underline{\operatorname{Proj}}_X (\bigoplus^\infty_{n=0} \mathcal{J}^n)} \cong \mathcal{O}_{\operatorname{Spec}(R[\frac{J}{d}])}$, the following square commutes
    \begin{displaymath}
        \begin{tikzcd}
            {(u^J_d)^\ast \mathcal{O}_{\underline{\operatorname{Proj}}_X (\bigoplus^\infty_{n=0} \mathcal{J}^n)}} & {(u^J_d)^\ast v_\ast \mathcal{O}_{\underline{\operatorname{Proj}}_X (\bigoplus^\infty_{n=0} \mathcal{P}^n)}} \\
            {\mathcal{O}_{\operatorname{Spec}(R[\frac{J}{d}])}} & {(v_d)_\ast \mathcal{O}_{\operatorname{Spec}(R[\frac{P}{d}])}.}
            \arrow["{(u^J_d)^\ast \eta_v}", from=1-1, to=1-2]
            \arrow["\cong"', from=1-1, to=2-1]
            \arrow["{\beta^J_d }", from=1-2, to=2-2]
            \arrow["{\eta_{v_d}}"', from=2-1, to=2-2]
        \end{tikzcd}
    \end{displaymath}
    Indeed, on an open $W\subseteq \operatorname{Spec}(R[\frac{J}{d}])$,
    both horizontal morphisms send a section in 
    \begin{displaymath}
        \mathcal{O}_{\underline{\operatorname{Proj}}_X (\bigoplus^\infty_{n=0} \mathcal{J}^n)} (W) = \mathcal{O}_{\operatorname{Spec}(R[\frac{J}{d}])} (W)
    \end{displaymath} 
    to the same section in 
    \begin{displaymath}
        \mathcal{O}_{\underline{\operatorname{Proj}}_X (\bigoplus^\infty_{n=0} \mathcal{P}^n)}(v^{-1}(W)) = \mathcal{O}_{\operatorname{Spec}(R[\frac{P}{d}])}(v^{-1}(W)).
    \end{displaymath}
    Thus, by the commutativity of the square, we complete the proof.
\end{proof}

\begin{remark}
    \label{rmk:represented_affine_blowup_algebras}
    For each $d\in \{a^3,ab^2,b^3,c r_1, \ldots, c r_e\}$, \cite[\href{https://stacks.math.columbia.edu/tag/07Z3}{Tag 07Z3}]{stacks-project} says $d$ is a nonzerodivisor in $R[\frac{J}{d}]$ and $R[\frac{J}{d}]_d \cong R_d$.
    Hence, the natural ring homomorphism $R[\frac{J}{d}] \to R[\frac{J}{d}]_d$ is injective.
    As $R$ is an integral domain, the natural ring homomorphism $R_d \to \operatorname{Frac}(R)$ is injective where $\frac{r}{d^n}\mapsto \frac{r}{d^n}$.
    Hence, $R[\frac{J}{d}]$ can be identified with a subring of $\operatorname{Frac}(R)$.
    Concretely, the elements of $R[\frac{J}{d}]$ can be written in the form $\frac{s}{d^n}$ with $s\in J^n$ \cite[\href{https://stacks.math.columbia.edu/tag/052P}{Tag 052P}]{stacks-project}. 
    Moreover, if $J=(g_1,\ldots,g_s)$, then $R[\frac{J}{d}]$ can be identified with $R [\frac{g_1}{d},\ldots, \frac{g_s}{d}] \subseteq \operatorname{Frac}(R)$. 
    See \cite[Remark 4.3]{Abad/Bravo/Villamayor:2020} or the discussion preceding \cite[Definition 5.6.2]{Swanson/Huneke:2006}.
    Similar reasoning gives analogous statements for $R[\frac{P}{d}]$.    
\end{remark}

\section{Vanishing}
\label{sec:local_vanishing}

\begin{lemma}
    \label{lem:vanishing_all_but_a3}
    Assume \Cref{setup:main}. 
    If $d\in \{ab^2,b^3,c r_1, \ldots, c r_e\}$, then 
    \begin{displaymath}
        \operatorname{cone}(\mathcal{O}_{\operatorname{Spec}(R[\frac{J}{d}])} \to \mathbf{R} v^d_\ast \mathcal{O}_{\operatorname{Spec}(R[\frac{P}{d}])}) \cong 0.
    \end{displaymath}
    In other words, the adjunction unit
    \begin{displaymath}
        \mathcal{O}_{\operatorname{Spec}(R[\frac{J}{d}])} \to \mathbf{R} v^d_\ast \mathcal{O}_{\operatorname{Spec}(R[\frac{P}{d}])}
    \end{displaymath}
    is an isomorphism.
\end{lemma}

\begin{proof}
    By \Cref{lem:concentrated_degree_zero_global}, the complex is concentrated in degree zero.
    Since $v^d$ is finite, the complex can be represented by a coherent sheaf.
    Applying \cite[\href{https://stacks.math.columbia.edu/tag/06Z0}{Tag 06Z0}]{stacks-project}, we can work with modules on the corresponding rings. 
    Hence, we want to show that the unit $R[\frac{J}{d}] \to R[\frac{P}{d}]$ is an isomorphism. 
    As an $R[\frac{J}{d}]$-algebra, $R[\frac{P}{d}]\cong R[\frac{J}{d}] [\frac{c}{d}]$. 
    By \cite[\href{https://stacks.math.columbia.edu/tag/0CC1}{Tag 0CC1}]{stacks-project} and \cite[\S 1, Proposition 3.9.9]{Grothendieck/Dieudonne:1971}, the cokernel of the unit $R[\frac{J}{d}] \to R[\frac{P}{d}]$ is $R[\frac{P}{d}]/R[\frac{J}{d}]$.
    Thus, it suffices to show that $\frac{c}{d} \in R[\frac{J}{d}]$.

    Observe that if there exists $N\geq 1$ such that $cd^{N-1}\in J^N$, then $\frac{c}{d}\in R[\frac{J}{d}]$. 
    Indeed, by \Cref{rmk:represented_affine_blowup_algebras}, $cd^{N-1}\in J^N$ implies $\frac{cd^{N-1}}{d^N}\in R[\frac{J}{d}]$, and so $\frac{c}{d}\in R[\frac{J}{d}]$.
    Then a direct computation yields 
    \begin{itemize}
        \item $cb^3 = (ab^2)^2 \in J^2$
        \item $c(ab^2)=a^3 b^3 \in J^2$
        \item $c (c r_i)^2 = r_i^2 (a^3)^2 b^3 \in J^3$ for each $i\in \{1,\ldots,e\}$.
    \end{itemize}
    This completes the proof.


\end{proof}

\begin{lemma}
    \label{lem:identifyJa3}
    Assume \Cref{setup:main}. 
    Then $R[\frac{J}{a^3}]$ and $R[\frac{P}{a^3}]$ respectively coincide with $R$-algebras  $R[(\frac{b}{a})^2, (\frac{b}{a})^3,\frac{br_1}{a}, \ldots, \frac{br_e}{a}]$ and $R[\frac{b}{a}]$.
\end{lemma}

\begin{proof}
    In $\operatorname{Frac}(R)$, one has $\frac{ab^2}{a^3} = (\frac{b}{a})^2$, $\frac{b^n}{a^n} = (\frac{b}{a})^n$, $\frac{cr_i}{a^3}=r_i \frac{b}{a}$, $\frac{c}{a^3} = \frac{b}{a}$.
    Then the claim follows.
    See \Cref{rmk:represented_affine_blowup_algebras}.

\end{proof}

\begin{lemma}
    \label{lem:ring_presentationJona3}
    Assume \Cref{setup:main}. 
    Let $T$ be an indeterminate. 
    There exists a surjective ring homomorphism $\phi\colon R[T] \to R[\frac{b}{a}]$ given by $T\mapsto \frac{b}{a}$.
    Moreover, the kernel of $\phi$ is the ideal $(aT - b)$, and hence,
    \begin{displaymath}
        R[T]/(aT - b) \cong R[\frac{b}{a}].
    \end{displaymath}
\end{lemma}

\begin{proof}
    The mapping is a well-defined ring homomorphism and is surjective by construction.
    Since $(aT-b) \subseteq \ker (\phi)$, we need to check the reverse inclusion. 
    Fix $F(T)\in \ker (\phi)$. 
    If $F(T)=0$, then there is nothing to check.
    Assume $F(T)\not=0$.
    Set $n:= \operatorname{deg} (F(T))$.
    We prove the claim by induction on $n$. 
    If $n=0$, then $F(T)\in R$ is mapped to zero in $\operatorname{Frac}(R)$, which implies $F(T)=0$.
    Now assume that $n\geq 1$.
    Write $F(T)=\sum^n_{i=0} r_i T^i$ with $r_n\not=0$.
    As $\sum^n_{i=0} r_i (\frac{b}{a})^i = 0$, multiplying by $a^n$ yields $\sum^n_{i=0} a^{n-i} r_i b^i = 0 \in R$. 
    Reducing modulo the principal ideal $(a)$, we obtain that $r_n b^n \in (a)$.
    However, $a$ and $b$ form an $R$-regular sequence, and so $r_n \in (a)$. 
    Then $r_n = s_n a$ for some $s_n\in R$. 
    It follows that $F(T)- s_n T^{n-1}(a T - b) \in \ker (\phi)$ and has degree strictly less than $n$. 
    Thus, by the induction step, $F(T)- s_n T^{n-1}(a T - b)$ is in the ideal $(a T - b)$. 
    Since $s_n T^{n-1}(a T - b)$ is in the ideal $(a T - b)$, it follows that $F(T)\in (a T - b)$.
\end{proof}

\begin{lemma}
    \label{lem:unit_for_open_immersion_with_support_is_isomorphism}
    Let $X$ be a Noetherian scheme with affine diagonal, $j\colon U \to X$ be an open immersion from an affine scheme, and $E\in D^b_{\operatorname{coh}}(X)$ satisfy $Z:=\operatorname{Supp}(E)\subseteq U$ set-theoretically.
    Denote by $\eta$ the unit for the adjoint pair $\mathbf{L}j^\ast$ and $\mathbf{R}j_\ast$.
    Then $\eta_E \colon E \to \mathbf{R}j_\ast \mathbf{L}j^\ast E$ is an isomorphism.
\end{lemma}

\begin{proof}
    Since $E$ is a bounded pseudocoherent complex, $\operatorname{Supp}(E)$ is closed.
    Set $V= X\setminus Z$.
    Denote by $s\colon V \to X$ the associated open immersion.
    Note that $U$ and $V$ form an open cover of $X$.
    Then $\mathbf{L}j^\ast \eta_E$ is an isomorphism (see e.g.\ \cite[Lemma A.4]{GuisadoVillaalgordo/Lank:2026} and \cite[\href{https://stacks.math.columbia.edu/tag/071Q}{Tag 071Q}]{stacks-project}, or \cite[Theorem 1]{Jorgensen:2009}).
    Hence, $\mathbf{L}j^\ast \operatorname{cone}(\eta_E)\cong 0$.
    As $\operatorname{Supp}(E)=Z$, it follows that $\mathbf{L}s^\ast E\cong 0$. 
    There exists a fibered square,
    \begin{displaymath}
        \begin{tikzcd}
            {V \cap U} & U \\
            V & {X.}
            \arrow["{s^\prime}", from=1-1, to=1-2]
            \arrow["{j^\prime}"', from=1-1, to=2-1]
            \arrow["j", from=1-2, to=2-2]
            \arrow["s"', from=2-1, to=2-2]
        \end{tikzcd}
    \end{displaymath}
    By flat base change, 
    \begin{displaymath}
        \mathbf{L}s^\ast \mathbf{R}j_\ast \mathbf{L}j^\ast E 
        \cong \mathbf{R}j^\prime_\ast \mathbf{L}(s^\prime)^\ast \mathbf{L}j^\ast E
        \cong \mathbf{R}j^\prime_\ast \mathbf{L}(j^\prime)^\ast \mathbf{L}s^\ast E \cong 0.
    \end{displaymath}
    This implies that $\mathbf{L}s^\ast \operatorname{cone}(\eta_E)\cong 0$.
    Consequently, $\operatorname{cone}(\eta_E)\cong 0$.
\end{proof}

\begin{lemma}
    \label{lem:unit_on_U_isomorphism}
    Assume \Cref{setup:main}. 
    There exists an isomorphism
    \begin{displaymath}
        \begin{aligned}
            \operatorname{cone}
            & (\mathcal{O}_{\underline{\operatorname{Proj}}_X (\bigoplus^\infty_{n=0} \mathcal{J}^n)} \to \mathbf{R}v_\ast \mathcal{O}_{\underline{\operatorname{Proj}}_X (\bigoplus^\infty_{n=0} \mathcal{P}^n)})
            \\&\xrightarrow{unit} \mathbf{R} (u^J_{a^3})_\ast  \mathbf{L} (u^J_{a^3})^\ast  \operatorname{cone}(\mathcal{O}_{\underline{\operatorname{Proj}}_X (\bigoplus^\infty_{n=0} \mathcal{J}^n)} \to \mathbf{R}v_\ast \mathcal{O}_{\underline{\operatorname{Proj}}_X (\bigoplus^\infty_{n=0} \mathcal{P}^n)}).
        \end{aligned}
    \end{displaymath}
\end{lemma}

\begin{proof}
    By \Cref{lem:concentrated_degree_zero_global},
    \begin{displaymath}
        \begin{aligned}
            \operatorname{coker}
            & (\mathcal{O}_{\underline{\operatorname{Proj}}_X (\bigoplus^\infty_{n=0} \mathcal{J}^n)} \to v_\ast \mathcal{O}_{\underline{\operatorname{Proj}}_X (\bigoplus^\infty_{n=0} \mathcal{P}^n)})
            \\&\cong \operatorname{cone}(\mathcal{O}_{\underline{\operatorname{Proj}}_X (\bigoplus^\infty_{n=0} \mathcal{J}^n)} \to \mathbf{R}v_\ast \mathcal{O}_{\underline{\operatorname{Proj}}_X (\bigoplus^\infty_{n=0} \mathcal{P}^n)}).
        \end{aligned}
    \end{displaymath}
    Moreover, \cite[\href{https://stacks.math.columbia.edu/tag/056J}{Tag 056J}]{stacks-project} says that the support of the cokernel is closed in $\underline{\operatorname{Proj}}_X (\oplus^\infty_{n=0} \mathcal{J}^n)$. 
    From \Cref{rmk:open_cover}, $\underline{\operatorname{Proj}}_X (\bigoplus^\infty_{n=0} \mathcal{J}^n)$ admits an open cover given by $\operatorname{Spec}(R[\frac{J}{d}])$ where $d\in \{a^3,ab^2,b^3,c r_1, \ldots, c r_e\}$.
    Applying \Cref{lem:vanishing_all_but_a3}, the cokernel vanishes on all $\operatorname{Spec}(R[\frac{J}{d}])$ where $d\in \{ab^2,b^3,c r_1, \ldots, c r_e\}$.
    Hence, the support of the cokernel is contained in $\operatorname{Spec}(R[\frac{J}{a^3}])$.
    Consequently, from \Cref{lem:unit_for_open_immersion_with_support_is_isomorphism},
    \begin{displaymath}
        \begin{aligned}
            \operatorname{cone}
            & (\mathcal{O}_{\underline{\operatorname{Proj}}_X (\bigoplus^\infty_{n=0} \mathcal{J}^n)} \to \mathbf{R}v_\ast \mathcal{O}_{\underline{\operatorname{Proj}}_X (\bigoplus^\infty_{n=0} \mathcal{P}^n)})
            \\&\cong \mathbf{R} (u^J_{a^3})_\ast  \mathbf{L} (u^J_{a^3})^\ast  \operatorname{cone}(\mathcal{O}_{\underline{\operatorname{Proj}}_X (\bigoplus^\infty_{n=0} \mathcal{J}^n)} \to \mathbf{R}v_\ast \mathcal{O}_{\underline{\operatorname{Proj}}_X (\bigoplus^\infty_{n=0} \mathcal{P}^n)}).
        \end{aligned}
    \end{displaymath}
    This completes the proof.
\end{proof}

\begin{remark}
    \label{rmk:support_on_chart}
    The proof of \Cref{lem:unit_on_U_isomorphism} shows that 
    \begin{equation}
        \label{eq:support_on_chart}
        S:= \operatorname{Supp}\big(\operatorname{cone} (\mathcal{O}_{\underline{\operatorname{Proj}}_X (\bigoplus^\infty_{n=0} \mathcal{J}^n)} \to \mathbf{R}v_\ast \mathcal{O}_{\underline{\operatorname{Proj}}_X (\bigoplus^\infty_{n=0} \mathcal{P}^n)})\big) \subseteq \operatorname{Spec}(R[\frac{J}{a^3}]).
    \end{equation}
\end{remark}

\begin{lemma}
    \label{lem:quotient_nonzero}
    Assume \Cref{setup:main}. 
    There exists a closed point $q\in \underline{\operatorname{Proj}}_X (\oplus^\infty_{n=0} \mathcal{J}^n)$ with associated closed immersion
    \begin{displaymath}
        h \colon \operatorname{Spec}(\kappa(q)) \to \underline{\operatorname{Proj}}_X (\bigoplus^\infty_{n=0} \mathcal{J}^n)
    \end{displaymath}
    such that 
    \begin{displaymath}
        \operatorname{cone}(\mathcal{O}_{\underline{\operatorname{Proj}}_X (\bigoplus^\infty_{n=0} \mathcal{J}^n)} \to \mathbf{R}v_\ast \mathcal{O}_{\underline{\operatorname{Proj}}_X (\bigoplus^\infty_{n=0} \mathcal{P}^n)})
        \cong \mathbf{R}h_\ast \mathcal{O}_{\operatorname{Spec}(\kappa(q))}.
    \end{displaymath}
\end{lemma}

\begin{proof}
    By \Cref{lem:unit_on_U_isomorphism,lem:concentrated_degree_zero_global},
    \begin{equation}
        \label{eq:coker_unit}
        \begin{aligned}
            \operatorname{cone}
            & (\mathcal{O}_{\underline{\operatorname{Proj}}_X (\bigoplus^\infty_{n=0} \mathcal{J}^n)} \to \mathbf{R}v_\ast \mathcal{O}_{\underline{\operatorname{Proj}}_X (\bigoplus^\infty_{n=0} \mathcal{P}^n)})
            \\&\cong \mathbf{R} (u^J_{a^3})_\ast  \mathbf{L} (u^J_{a^3})^\ast  \operatorname{cone}(\mathcal{O}_{\underline{\operatorname{Proj}}_X (\bigoplus^\infty_{n=0} \mathcal{J}^n)} \to \mathbf{R}v_\ast \mathcal{O}_{\underline{\operatorname{Proj}}_X (\bigoplus^\infty_{n=0} \mathcal{P}^n)})
            \\&\cong \mathbf{R} (u^J_{a^3})_\ast \operatorname{coker} (\mathcal{O}_{\operatorname{Spec}(R[\frac{J}{a^3}])} \to v^{a^3}_\ast \mathcal{O}_{\operatorname{Spec}(R[\frac{P}{a^3}])}).
        \end{aligned}
    \end{equation}
    We will prove that there exists a closed immersion $h_q\colon \operatorname{Spec}(\kappa(q)) \to \operatorname{Spec}(R[\frac{J}{a^3}])$ such that $u^J_{a^3} \circ h_q$ is a closed immersion and that 
    \begin{displaymath}
        \operatorname{coker} (\mathcal{O}_{\operatorname{Spec}(R[\frac{J}{a^3}])} \to v^{a^3}_\ast \mathcal{O}_{\operatorname{Spec}(R[\frac{P}{a^3}])}) \cong \mathbf{R}(h_q)_\ast \mathcal{O}_{\operatorname{Spec}(\kappa(q))}.
    \end{displaymath}
    Applying \cite[\href{https://stacks.math.columbia.edu/tag/06Z0}{Tag 06Z0}]{stacks-project}, we can work with modules over the corresponding rings. 

    Let $T$ be an indeterminate. 
    Consider the ring homomorphism $\epsilon^\prime \colon R[T] \to R/\mathfrak{m}$ given by the composition of ring homomorphisms,
    \begin{displaymath}
        R[T] \xrightarrow{T\mapsto 0} R \to R/\mathfrak{m}.
    \end{displaymath}
    Here $R \to R/\mathfrak{m}$ is the quotient ring homomorphism.
    Since $(aT-b)\subseteq \ker (\epsilon^\prime)$, there exists a unique ring homomorphism $\epsilon \colon R[T]/(aT-b) \to R/\mathfrak{m}$ such that $\epsilon^\prime$ is the composition of the quotient ring homomorphism $R[T] \to R[T]/(aT-b)$ and $\epsilon$.
    As $R \subseteq R[T]$ is a subring and $R\subseteq R[T] \xrightarrow{T \mapsto 0} R$ is the identity, it follows that $\epsilon$ is surjective.

    Set $t:= \frac{b}{a}$. 
    There exists a composition of ring homomorphisms
    \begin{equation}
        \label{eq:factor_affine_blowup_algebras}
        \theta \colon R[t^2, t^3, r_1 t, \ldots, r_e t] 
        \subseteq R[t] 
        \cong R[T]/(aT-b)
        \xrightarrow{\epsilon} R / \mathfrak{m}
    \end{equation}
    where the isomorphism is from \Cref{lem:ring_presentationJona3}.
    Since $R\subseteq R[t^2, t^3,r_1 t, \ldots, r_e t]$, $\theta$ is surjective.
    The kernel of $\theta$ is the ideal $\mathfrak{q} = (\mathfrak{m}$, $\mathfrak{m} t, t^2, t^3)$ in $R[t^2, t^3,r_1 t, \ldots, r_e t]$.
    Note that $\mathfrak{q}\subseteq \ker (\theta)$.
    Every element of $R[t^2, t^3,r_1 t, \ldots, r_e t]$ can be expressed as a polynomial in $t^2, t^3,r_1 t, \ldots, r_e t$ with coefficients in $R$.
    Hence, as an additive group, 
    \begin{displaymath}
        R[t^2, t^3,r_1 t, \ldots, r_e t] = R + (t^2, t^3,r_1 t, \ldots, r_e t).
    \end{displaymath}
    Choose $F\in \ker (\theta)$.
    Write $F=F_1 + F_2$ with $F_1\in R$ and $F_2\in (t^2, t^3,r_1 t, \ldots, r_e t)$. 
    If $\theta (F)= 0 + \mathfrak{m}$, then $F_1\in \mathfrak{m}$.
    Hence, $F\in (\mathfrak{m}) + (t^2, t^3,r_1 t, \ldots, r_e t) = \mathfrak{q}$, and so $\mathfrak{q} = \ker (\theta)$. 

    Since $t^n \in R[t^2, t^3,r_1 t, \ldots, r_e t]$ for $n\geq 2$, we have that $R[t]$ is generated by $1$ and $t$ as an $R[t^2, t^3,r_1 t, \ldots, r_e t]$-module.
    Hence, the coset $\overline{t} := t+ R[t^2, t^3,r_1 t, \ldots, r_e t]$ generates the quotient module
    \begin{displaymath}
        R[t]/ R[t^2, t^3,r_1 t, \ldots, r_e t].
    \end{displaymath}
    Moreover, 
    \begin{displaymath}
        \mathfrak{m}t, 
        \quad (\mathfrak{m}t)t = \mathfrak{m}t^2,
        \quad t^2 t = t^3,
        \quad t^3 t = t^4
    \end{displaymath}
    belong to $R[t^2, t^3,r_1 t, \ldots, r_e t]$.
    Therefore, $\mathfrak{q}$ annihilates the quotient module.

    It follows that the surjection of $R[t^2, t^3,r_1 t, \ldots, r_e t]$-modules
    \begin{displaymath}
        R[t^2, t^3,r_1 t, \ldots, r_e t] \to R[t]/R[t^2, t^3,r_1 t, \ldots, r_e t], 
        \quad 1 \mapsto \overline{t}
    \end{displaymath}
    factors through a surjection 
    \begin{equation}
        \label{eq:residue_field_to_quotient}
        R[t^2, t^3,r_1 t, \ldots, r_e t]/\mathfrak{q} \cong k \to R[t]/R[t^2, t^3,r_1 t, \ldots, r_e t].
    \end{equation}
    We show that this quotient is nonzero.
    Using \Cref{lem:ring_presentationJona3}, we have a composition of ring homomorphisms
    \begin{displaymath}
        R[t^2, t^3, r_1 t, \ldots, r_e t] \to R[t] \cong R[T]/(aT-b).
    \end{displaymath}
    Consider the ideal $\mathfrak{m}R[T]$ and the natural ring isomorphism 
    \begin{displaymath}
        R[T]/\mathfrak{m}R[T] \to (R/\mathfrak{m})[T].
    \end{displaymath}
    Since $(aT-b)\in \mathfrak{m}R[T]$, there exists a quotient ring homomorphism 
    \begin{displaymath}
        R[T]/(aT-b) \to \frac{R[T]/(aT-b)}{\mathfrak{m}R[T]/(aT-b)}.
    \end{displaymath}
    However, this quotient ring is isomorphic to $R[T]/\mathfrak{m}R[T]$. 
    Then $t^2, t^3,r_1 t, \ldots, r_e t$ are mapped to cosets which are respectively represented by $T^2,T^3,r_1T,\ldots,r_eT$. 
    Now, after composing with $R[T]/\mathfrak{m}R[T] \to (R/\mathfrak{m})[T]$, their images correspond respectively to $T^2, T^3, 0,\ldots, 0$, so the image is $(R/\mathfrak{m})[T^2,T^3]$.
    Since $T\not\in (R/\mathfrak{m})[T^2,T^3]$, it follows that $t\not\in R[t^2, t^3, r_1 t, \ldots, r_e t]$.
    Hence, $R[t]/R[t^2, t^3,r_1 t, \ldots, r_e t]$ is nonzero.

    As $\mathfrak{q}$ is maximal, $k$ is a simple $R[t^2, t^3,r_1 t, \ldots, r_e t]$-module.
    Hence, the image of \eqref{eq:residue_field_to_quotient} being nonzero implies its kernel is zero.
    Thus, \eqref{eq:residue_field_to_quotient} is an isomorphism.
    Moreover, \eqref{eq:factor_affine_blowup_algebras} yields an isomorphism of $R[t^2, t^3,r_1 t, \ldots, r_e t]$-algebras
    \begin{equation}
        \label{eq:residue_field_isomorphism}
        R[t^2, t^3,r_1 t, \ldots, r_e t] / \mathfrak{q} \to k.
    \end{equation}

    Let $q$ be the closed point of $\operatorname{Spec}(R[\frac{J}{a^3}])$ corresponding to the maximal ideal $\mathfrak{q}$.
    Denote by $h_q\colon \operatorname{Spec}(\kappa(q)) \to \operatorname{Spec}(R[\frac{J}{a^3}])$ the associated closed immersion.
    The composition $u^J_{a^3} \circ h_q$ is an immersion.
    We claim that $u^J_{a^3} \circ h_q$ is a closed immersion.
    Applying \cite[\href{https://stacks.math.columbia.edu/tag/01IQ}{Tag 01IQ}]{stacks-project}, it suffices to check that the topological image of $u^J_{a^3} \circ h_q$ is closed.
    By \Cref{rmk:support_on_chart}, $S\subseteq \operatorname{Spec}(R[\frac{J}{a^3}])$ (see \eqref{eq:support_on_chart}).
    By \cite[\href{https://stacks.math.columbia.edu/tag/056J}{Tag 056J}]{stacks-project}, $S$ coincides with the collection of $p\in \underline{\operatorname{Proj}}_X (\bigoplus^\infty_{n=0} \mathcal{J}^n)$ such that 
    \begin{displaymath}
        \big( \operatorname{coker}(\mathcal{O}_{\underline{\operatorname{Proj}}_X (\bigoplus^\infty_{n=0} \mathcal{J}^n)} \to v_\ast \mathcal{O}_{\underline{\operatorname{Proj}}_X (\bigoplus^\infty_{n=0} \mathcal{P}^n)}) \big)_p \otimes_{\mathcal{O}_{\underline{\operatorname{Proj}}_X (\bigoplus^\infty_{n=0} \mathcal{J}^n),p}} \kappa(p) \not\cong 0.
    \end{displaymath}
    Hence, for any $p\in S$, \eqref{eq:support_on_chart} implies the natural morphism $\operatorname{Spec}(\kappa(p)) \to \underline{\operatorname{Proj}}_X (\bigoplus^\infty_{n=0} \mathcal{J}^n)$ factors through $u^J_{a^3}$. 
    Then, from \eqref{eq:concentrated_degree_zero_global}, it follows that
    \begin{displaymath}
        S= u^J_{a^3} \big( \operatorname{Supp}\big( \operatorname{coker} (\mathcal{O}_{\operatorname{Spec}(R[\frac{J}{a^3}])} \to v^{a^3}_\ast \mathcal{O}_{\operatorname{Spec}(R[\frac{P}{a^3}])}) \big) \big).
    \end{displaymath} 
    However, \cite[\href{https://stacks.math.columbia.edu/tag/06Z0}{Tag 06Z0}]{stacks-project}, \eqref{eq:residue_field_isomorphism}, and \eqref{eq:residue_field_to_quotient} imply 
    \begin{equation}
        \label{eq:decompose_direct_sum_of_fields}
        \operatorname{coker} (\mathcal{O}_{\operatorname{Spec}(R[\frac{J}{a^3}])} \to v^{a^3}_\ast \mathcal{O}_{\operatorname{Spec}(R[\frac{P}{a^3}])}) \cong \mathbf{R}(h_q)_\ast \mathcal{O}_{\operatorname{Spec}(\kappa(q))}.
    \end{equation}
    By \cite[Remark 23.46(2)]{Gortz/Wedhorn:2023}, the support of the complex above is $\{q\}$.
    Thus, $q$ is a closed point in $\underline{\operatorname{Proj}}_X (\bigoplus^\infty_{n=0} \mathcal{J}^n)$.

    Now, by combining \eqref{eq:coker_unit} and \eqref{eq:decompose_direct_sum_of_fields},
    \begin{equation}
        \label{eq:quotient_nonzero_final}
        \begin{aligned}
            \operatorname{cone}
            & (\mathcal{O}_{\underline{\operatorname{Proj}}_X (\bigoplus^\infty_{n=0} \mathcal{J}^n)} \to \mathbf{R}v_\ast \mathcal{O}_{\underline{\operatorname{Proj}}_X (\bigoplus^\infty_{n=0} \mathcal{P}^n)}) 
            \\& \cong \mathbf{R}(u^J_{a^3}\circ h_q)_\ast \mathcal{O}_{\operatorname{Spec}(\kappa(q))}.
        \end{aligned}
    \end{equation}
    This completes the proof.
\end{proof}

\section{Regularity}
\label{sec:regularity}

\begin{lemma}
    \label{lem:blowup_extensions_for_residue_field}
    Assume \Cref{setup:main}. 
    There exists a distinguished triangle  
    \begin{displaymath}
        \begin{aligned}
            \mathbf{R}(f_\mathcal{J})_\ast \mathcal{O}_{\underline{\operatorname{Proj}}_X (\bigoplus^\infty_{n=0} \mathcal{J}^n)} 
            &\to \mathbf{R}(f_\mathcal{P})_\ast \mathcal{O}_{\underline{\operatorname{Proj}}_X (\bigoplus^\infty_{n=0} \mathcal{P}^n)}
            \\& \to \mathbf{R} i_\ast \mathcal{O}_{\operatorname{Spec}(k)} 
            \to \mathbf{R}(f_\mathcal{J})_\ast \mathcal{O}_{\underline{\operatorname{Proj}}_X (\bigoplus^\infty_{n=0} \mathcal{J}^n)} [1].
        \end{aligned}
    \end{displaymath}
\end{lemma}

\begin{proof}
    As $f_{\mathcal{J}}$ is a closed morphism, it follows that $(f_{\mathcal{J}} \circ u^J_{a^3} \circ h_q)(q)$ is a closed point in $X$. 
    However, $R$ is a local ring, and so $(f_{\mathcal{J}} \circ u^J_{a^3} \circ h_q)(q) = \mathfrak{m}$. 
    Hence, there exists a morphism $w$ fitting into a commutative diagram (see \cite[\href{https://stacks.math.columbia.edu/tag/01J5}{Tag 01J5}]{stacks-project})
    \begin{displaymath}
        \begin{tikzcd}
            {\operatorname{Spec}(\kappa(q))} & {\operatorname{Spec}(k)} \\
            {\underline{\operatorname{Proj}}_X (\bigoplus^\infty_{n=0} \mathcal{J}^n)} & {X.}
            \arrow["w", from=1-1, to=1-2]
            \arrow["{u^J_{a^3} \circ h_q}"', from=1-1, to=2-1]
            \arrow["i", from=1-2, to=2-2]
            \arrow["{f_{\mathcal{J}}}"', from=2-1, to=2-2]
        \end{tikzcd}
    \end{displaymath}
    We claim that $w$ is an isomorphism. 
    By the construction of $\theta$ in \eqref{eq:factor_affine_blowup_algebras}, its restriction to $R$ is the quotient ring homomorphism $R\to k$ and $\ker (\theta) = \mathfrak{q}$.
    Thus, $\theta$ induces an isomorphism of $R$-algebras $R[\frac{J}{a^3}]/\mathfrak{q} \to k$.
    Consequently, the natural ring homomorphism $k \to \kappa(q)=R[\frac{J}{a^3}]/\mathfrak{q}$ is an isomorphism, and hence so is $w$. 
    Therefore, from \eqref{eq:finite_birational} and \eqref{eq:quotient_nonzero_final}, it follows that 
    \begin{displaymath}
        \begin{aligned}
            \operatorname{cone}
            & (\mathbf{R}(f_\mathcal{J})_\ast \mathcal{O}_{\underline{\operatorname{Proj}}_X (\bigoplus^\infty_{n=0} \mathcal{J}^n)} \to \mathbf{R}(f_\mathcal{P})_\ast \mathcal{O}_{\underline{\operatorname{Proj}}_X (\bigoplus^\infty_{n=0} \mathcal{P}^n)}) 
            \\& \cong \mathbf{R}(f_{\mathcal{J}} \circ u^J_{a^3} \circ h_q)_\ast \mathcal{O}_{\operatorname{Spec}(\kappa(q))}
            \\&\cong \mathbf{R}i_\ast \mathcal{O}_{\operatorname{Spec}(k)}.
        \end{aligned}
    \end{displaymath}
    This completes the proof.
\end{proof}

\begin{lemma}
    \label{lem:section_retract_adjunction}
    Let $L \colon \mathcal{C} \rightleftarrows \mathcal{D} \colon R$ be an adjoint pair of functors between categories. 
    If $c$ is a retract of $R (d)$ in $\mathcal{C}$, then the unit map $\eta_c \colon c \to R L (c)$ is a split monomorphism in $\mathcal{C}$.
    A similar statement can be made for the counit being a split epimorphism.
\end{lemma}

\begin{proof}
    This is known but we write the proof.
    Choose $f\colon c \to R (d)$ with retraction $g \colon R(d) \to c$, i.e.\ $\operatorname{id}_{c} = g\circ f$. 
    By adjunction, $f$ factors through $\eta_c$, i.e.\ $f = F \circ \eta_c$. 
    Then $\operatorname{id}_c = (g \circ F) \circ \eta_c$.
    This shows the desired claim. 
\end{proof}

\begin{proof}
    [Proof of \Cref{prop:blowup_extensions_for_residue_field}]
    We prove the first claim. 
    Set $Q := (a,b)$.
    Denote by $\mathcal{Q}$ the ideal sheaf obtained by the sheafification of $Q$. 
    Given a morphism $g\colon Z\to W$ of schemes, let $\eta_g$ denote the unit $\mathcal{O}_W \to \mathbf{R}g_\ast \mathcal{O}_Z$. 
    By \cite[Exercise 13.22]{Gortz/Wedhorn:2020}, there exists an $X$-isomorphism of schemes 
    \begin{displaymath}
        f\colon \underline{\operatorname{Proj}}_X (\bigoplus^\infty_{n=0} \mathcal{Q}^n) \to \underline{\operatorname{Proj}}_X (\bigoplus^\infty_{n=0} \mathcal{P}^n).
    \end{displaymath}
    This yields 
    \begin{displaymath}
        \mathbf{R}(f_\mathcal{P})_\ast (\eta_f) \colon \mathbf{R}(f_\mathcal{P})_\ast \mathcal{O}_{\underline{\operatorname{Proj}}_X (\bigoplus^\infty_{n=0} \mathcal{P}^n)} 
        \to \mathbf{R}(f_\mathcal{Q})_\ast \mathcal{O}_{\underline{\operatorname{Proj}}_X (\bigoplus^\infty_{n=0} \mathcal{Q}^n)}.
    \end{displaymath}
    Since $a$ and $b$ form an $R$-regular sequence, \cite[\href{https://stacks.math.columbia.edu/tag/063L}{Tag 063L}]{stacks-project} says the closed immersion associated to $\mathcal{Q}$ is Koszul-regular.
    Then, by \cite[Lemma 2.3]{Thomason:1993}, we have an isomorphism,
    \begin{displaymath}
       \eta_{f_\mathcal{Q}}\colon \mathcal{O}_X \to \mathbf{R}(f_\mathcal{Q})_\ast \mathcal{O}_{\underline{\operatorname{Proj}}_X (\bigoplus^\infty_{n=0} \mathcal{Q}^n)}.
    \end{displaymath}
    Hence, the distinguished triangle from \Cref{lem:blowup_extensions_for_residue_field} can be rewritten as 
    \begin{displaymath}
        \mathbf{R}(f_\mathcal{J})_\ast \mathcal{O}_{\underline{\operatorname{Proj}}_X (\bigoplus^\infty_{n=0} \mathcal{J}^n)} \to \mathcal{O}_X \to \mathbf{R}i_\ast \mathcal{O}_{\operatorname{Spec}(k)} \to  \mathbf{R}(f_\mathcal{J})_\ast \mathcal{O}_{\underline{\operatorname{Proj}}_X (\bigoplus^\infty_{n=0} \mathcal{J}^n)}[1].
    \end{displaymath}
    We show that there exists an isomorphism
    \begin{displaymath}
        \mathbf{R}(f_\mathcal{J})_\ast \mathcal{O}_{\underline{\operatorname{Proj}}_X (\bigoplus^\infty_{n=0} \mathcal{J}^n)} \cong \mathbf{R}i_\ast \mathcal{O}_{\operatorname{Spec}(k)}[-1] \oplus \mathcal{O}_X.
    \end{displaymath}
    Observe that there exists a commutative diagram 
    \begin{displaymath}
        \begin{tikzcd}
            {\mathcal{O}_X} & {\mathbf{R}(f_\mathcal{J})_\ast \mathcal{O}_{\underline{\operatorname{Proj}}_X (\bigoplus^\infty_{n=0} \mathcal{J}^n)}} \\
            & {\mathbf{R}(f_\mathcal{P})_\ast \mathcal{O}_{\underline{\operatorname{Proj}}_X (\bigoplus^\infty_{n=0} \mathcal{P}^n)}} \\
            & {\mathbf{R}(f_\mathcal{Q})_\ast \mathcal{O}_{\underline{\operatorname{Proj}}_X (\bigoplus^\infty_{n=0} \mathcal{Q}^n)}.}
            \arrow["{\eta_{f_\mathcal{J}}}", from=1-1, to=1-2]
            \arrow["{\eta_{f_\mathcal{P}}}"', from=1-1, to=2-2]
            \arrow["{\mathbf{R}(f_\mathcal{J})_\ast \eta_v}", from=1-2, to=2-2]
            \arrow["{\mathbf{R}(f_\mathcal{P})_\ast \eta_f}"', from=2-2, to=3-2]
        \end{tikzcd}
    \end{displaymath}
    Since $\eta_{f_\mathcal{Q}} = \mathbf{R}(f_\mathcal{P})_\ast \eta_f \circ \eta_{f_\mathcal{P}}$ is an isomorphism, \Cref{lem:section_retract_adjunction} says that $\eta_{f_\mathcal{J}}$ splits. 
    Apply the octahedral axiom to $\mathbf{R}(f_\mathcal{P})_\ast \eta_f \circ \mathbf{R}(f_\mathcal{J})_\ast \eta_v$.
    By \Cref{lem:blowup_extensions_for_residue_field}, the cone of the first morphism is $\mathbf{R}i_\ast \mathcal{O}_{\operatorname{Spec}(k)}$.
    On the other hand, the cone of the second morphism is the zero object.
    Hence, we obtain 
    \begin{displaymath}
        \operatorname{cone}(\mathbf{R}(f_\mathcal{P})_\ast \eta_f \circ \mathbf{R}(f_\mathcal{J})_\ast \eta_v)
        \cong \mathbf{R}i_\ast \mathcal{O}_{\operatorname{Spec}(k)}.
    \end{displaymath}
    Again apply the octahedral axiom to 
    \begin{displaymath}
        \begin{tikzcd}
            {\mathcal{O}_X} & {\mathbf{R}(f_\mathcal{J})_\ast \mathcal{O}_{\underline{\operatorname{Proj}}_X (\bigoplus^\infty_{n=0} \mathcal{J}^n)}} \\
            & {\mathbf{R}(f_\mathcal{Q})_\ast \mathcal{O}_{\underline{\operatorname{Proj}}_X (\bigoplus^\infty_{n=0} \mathcal{Q}^n)}.}
            \arrow["{\eta_{f_\mathcal{J}}}", from=1-1, to=1-2]
            \arrow["{\eta_{f_\mathcal{Q}}}"', from=1-1, to=2-2]
            \arrow["{\mathbf{R}(f_\mathcal{P})_\ast \eta_f \circ \mathbf{R}(f_\mathcal{J})_\ast \eta_v}", from=1-2, to=2-2]
        \end{tikzcd}
    \end{displaymath}
    The composite $\mathbf{R}(f_\mathcal{P})_\ast \eta_f \circ \eta_{f_\mathcal{P}}$ is $\eta_{f_\mathcal{Q}}$, which is an isomorphism.
    Hence, we obtain that 
    \begin{displaymath}
        \operatorname{cone}(\eta_{f_\mathcal{J}}) 
        \cong  \operatorname{cone}(\mathbf{R}(f_\mathcal{P})_\ast \eta_f \circ \mathbf{R}(f_\mathcal{J})_\ast \eta_v)[-1] 
        \cong \mathbf{R}i_\ast \mathcal{O}_{\operatorname{Spec}(k)} [-1].
    \end{displaymath}
    Thus, the first claim follows.

    Now, we prove the second claim. 
    If $X$ is regular, then $D^b_{\operatorname{coh}}(X)=\operatorname{Perf}(X)$ \cite[\href{https://stacks.math.columbia.edu/tag/0FDC}{Tag 0FDC}]{stacks-project}.
    Hence, the derived pushforward of the structure sheaf along the blowup of $J$ is perfect.
    We prove the converse.
    Assume that the derived pushforward of the structure sheaf along the blowup of $J$ is perfect.
    By the first claim, $\mathbf{R}i_\ast \mathcal{O}_{\operatorname{Spec}(k)}[-1]$ is a direct summand of a perfect complex.
    Since $\operatorname{Perf}(X)$ is closed under direct summands and shifts, $\mathbf{R}i_\ast \mathcal{O}_{\operatorname{Spec}(k)}$ is a perfect complex. 
    Applying \cite[\href{https://stacks.math.columbia.edu/tag/066Q}{Tags 066Q} \& \href{https://stacks.math.columbia.edu/tag/00OC}{00OC}]{stacks-project}, we are done.
\end{proof}

\begin{lemma}
    \label{lem:perfect_finite_modification}
    Let $f\colon Y:=\operatorname{Spec}(S) \to X:=\operatorname{Spec}(R)$ be a finite birational morphism of integral Noetherian schemes. 
    If $\mathbf{R}f_\ast \mathcal{O}_Y \in \operatorname{Perf}(X)$, then the unit $\mathcal{O}_X \to \mathbf{R}f_\ast \mathcal{O}_Y$ is an isomorphism.
\end{lemma}

\begin{proof}
    By \cite[Lemma 3.16]{Lank/Venkatesh:2026} (or equivalently, \cite[Lemma 2.4]{DeDeyn/Lank/Lank/ManaliRahul/Venkatesh:2026}), the unit $\mathcal{O}_X \to \mathbf{R}f_\ast \mathcal{O}_Y$ splits.
    Moreover, $f$ is finite, and so \cite[\href{https://stacks.math.columbia.edu/tag/0G9R}{Tag 0G9R}]{stacks-project} says the canonical morphism $f_\ast \mathcal{O}_Y \to \mathbf{R}f_\ast \mathcal{O}_Y$ is an isomorphism.
    Hence, the morphism $\mathcal{O}_X \to f_\ast \mathcal{O}_Y$ splits.
    Since $f$ is a proper birational morphism of integral Noetherian schemes, \cite[Proposition 7.4.5]{Grothendieck/Dieudonne:1971} implies that $f_\ast \mathcal{O}_Y$ is torsion free.
    Moreover, $\mathcal{O}_X \to f_\ast \mathcal{O}_Y$ is generically an isomorphism because $f$ is birational.
    Thus, the cokernel of $\mathcal{O}_X \to f_\ast \mathcal{O}_Y$ is a torsion $\mathcal{O}_X$-module. 
    On the other hand, the splitting $\mathcal{O}_X \to f_\ast \mathcal{O}_Y$  yields that $f_\ast \mathcal{O}_Y / \mathcal{O}_X$ is a direct summand of the torsion free sheaf $f_\ast \mathcal{O}_Y$.
    Consequently, $f_\ast \mathcal{O}_Y / \mathcal{O}_X \cong 0$, which completes the proof.
\end{proof}

\begin{proposition}
    \label{prop:characterize_regularity_local}
    \Cref{thm:characterize_regularity} is true if $R$ is a local ring.
\end{proposition}

\begin{proof}
    If $X$ is regular, then $D^b_{\operatorname{coh}}(X)=\operatorname{Perf}(X)$ \cite[\href{https://stacks.math.columbia.edu/tag/0FDC}{Tag 0FDC}]{stacks-project}.
    Hence, $\mathbf{R}f_\ast \mathcal{O}_Y$ is perfect for all blowups $f\colon Y \to X$.
    We prove the converse.

    First, we claim that $X$ is normal. 
    Choose $x\in\operatorname{Frac}(R)$ integral over $R$.
    Set $S:=R[x] \subseteq \operatorname{Frac}(R)$. 
    Note that $S$ is an integral domain. 
    Denote by $f\colon \operatorname{Spec}(S)\to \operatorname{Spec}(R)$ the morphism induced by the ring homomorphism $R \to S$. 
    The morphism $f$ is finite, hence projective, and birational.
    By \cite[Theorem 8.1.24]{Liu:2002}, $f$ is a blowup.
    Then $\mathbf{R}f_\ast \mathcal{O}_{\operatorname{Spec}(S)} \in \operatorname{Perf}(X)$, and so \Cref{lem:perfect_finite_modification} implies $R = R[x]$.
    Consequently, every element of $\operatorname{Frac}(R)$ integral over $R$ belongs to $R$.
    Thus, $R$ is normal.

    Now, we finish the proof.
    If $\dim X \leq 1$, then the ring is regular because it is normal. 
    If $\dim X > 1$, then \cite[\href{https://stacks.math.columbia.edu/tag/031S}{Tag 031S}]{stacks-project} implies $\operatorname{depth}(R)\geq 2$.
    By \Cref{prop:blowup_extensions_for_residue_field}, there exists an ideal sheaf $\mathcal{J}\subseteq \mathcal{O}_X$ such that 
    \begin{displaymath}
        \mathbf{R}(f_\mathcal{J})_\ast \mathcal{O}_{\underline{\operatorname{Proj}}_X (\bigoplus^\infty_{n=0} \mathcal{J}^n)} 
        \cong \mathbf{R} i_\ast \mathcal{O}_{\operatorname{Spec}(k)} [-1]\oplus \mathcal{O}_X.
    \end{displaymath}
    Since $\mathbf{R}(f_\mathcal{J})_\ast \mathcal{O}_{\underline{\operatorname{Proj}}_X (\bigoplus^\infty_{n=0} \mathcal{J}^n)}$ is perfect, \Cref{prop:blowup_extensions_for_residue_field} yields that $R$ is regular. 
\end{proof}

\begin{proof}
    [Proof of \Cref{thm:characterize_regularity}]
    As in the proof of \Cref{prop:characterize_regularity_local}, it remains to prove the converse.

    Let $\mathfrak{p} \in X$. 
    Write $s_\mathfrak{p} \colon \operatorname{Spec}(R_{\mathfrak{p}}) \to X$ for the morphism induced by the localization homomorphism $s^\#_\mathfrak{p} \colon R \to R_\mathfrak{p}$. 
    As $R$ is integral, $s^\#_\mathfrak{p}$ is injective. 
    Choose an ideal $A\subseteq R_\mathfrak{p}$.
    Set $B:=(s^\#_\mathfrak{p})^{-1}(A)$.
    Then $A$ is an ideal of $R_\mathfrak{p}$ generated by the image of $B$ (i.e.\ $A = B R_\mathfrak{p}$).
    See \cite[\href{https://stacks.math.columbia.edu/tag/02C9}{Tag 02C9}]{stacks-project}.

    By \cite[\href{https://stacks.math.columbia.edu/tag/0805}{Tag 0805}]{stacks-project}, blowups are compatible with flat base change. 
    Denote by $\mathcal{B}$ and $\mathcal{A}$ respectively the sheafifications of $B$ and $A$.
    This yields an isomorphism
    \begin{displaymath}
        \underline{\operatorname{Proj}}_X (\bigoplus^\infty_{n=0} \mathcal{B}^n) \times_X \operatorname{Spec}(\mathcal{O}_{X,\mathfrak{p}}) 
        \cong \underline{\operatorname{Proj}}_{\operatorname{Spec}(R_{\mathfrak{p}}) } (\bigoplus^\infty_{n=0} \mathcal{A}^n).
    \end{displaymath}
    If $g_\mathcal{A} \colon  \underline{\operatorname{Proj}}_{\operatorname{Spec}(R_{\mathfrak{p}}) } (\bigoplus^\infty_{n=0} \mathcal{A}^n) \to \operatorname{Spec}(R_\mathfrak{p})$ denotes the blowup of $\operatorname{Spec}(R_\mathfrak{p})$ along $\mathcal{A}$, then flat base change yields 
    \begin{displaymath}
        \mathbf{L}s_{\mathfrak{p}}^\ast \mathbf{R} (f_\mathcal{B})_\ast \mathcal{O}_{\underline{\operatorname{Proj}}_X (\bigoplus^\infty_{n=0} \mathcal{B}^n)} 
        \cong \mathbf{R} (g_\mathcal{A})_\ast \mathcal{O}_{\underline{\operatorname{Proj}}_{\operatorname{Spec}(R_\mathfrak{p})} (\bigoplus^\infty_{n=0} \mathcal{A}^n)}.
    \end{displaymath}

    The hypothesis says that 
    \begin{displaymath}
        \mathbf{R} (f_\mathcal{B})_\ast \mathcal{O}_{\underline{\operatorname{Proj}}_X (\bigoplus^\infty_{n=0} \mathcal{B}^n)} \in \operatorname{Perf}(X).
    \end{displaymath}
    By \cite[\href{https://stacks.math.columbia.edu/tag/09UA}{Tag 09UA}]{stacks-project}, $\mathbf{L}s_{\mathfrak{p}}^\ast \mathbf{R} (f_\mathcal{B})_\ast \mathcal{O}_{\underline{\operatorname{Proj}}_X (\bigoplus^\infty_{n=0} \mathcal{B}^n)}$ is perfect. 
    Thus, 
    \begin{displaymath}
        \mathbf{R} (g_\mathcal{A})_\ast \mathcal{O}_{\underline{\operatorname{Proj}}_{\operatorname{Spec}(R_\mathfrak{p})} (\bigoplus^\infty_{n=0} \mathcal{A}^n)}
    \end{displaymath}
    is perfect. 
    By \Cref{prop:characterize_regularity_local}, $R_\mathfrak{p}$ is regular because $A$ was an arbitrary ideal.
    Consequently, as $\mathfrak{p}$ was arbitrary, $R$ is regular.
\end{proof}

\bibliographystyle{alpha}
\bibliography{mainbib}

\end{document}